\documentclass{article}

\usepackage{amsmath,amssymb,amsfonts}%
\usepackage{theorem}%
\usepackage{hyperref}%

\theoremstyle{change}%

\def\MR#1#2{\href{http://www.ams.org/mathscinet-getitem?mr=#1}{#2}}%

\newtheorem{definition}{Definition:}[section]%
\newtheorem{theorem}[definition]{Theorem:}%
\newtheorem{lemma}[definition]{Lemma:}%
\newtheorem{proposition}[definition]{Proposition:}%
\newtheorem{corollary}[definition]{Corollary:}%

{\theorembodyfont{\rmfamily} \newtheorem{remark}[definition]{Remark:}}%
{\theorembodyfont{\rmfamily} \newtheorem{example}[definition]{Example:}}%

\newenvironment{proof}
  {{\bf Proof:}}
  {\qquad \hspace*{\fill} $\Box$}

\newcommand{\N}{\mathbb{N}}%
\newcommand{\R}{\mathbb{R}}%
\renewcommand{\P}{\mathbb{P}}%

\newcommand{\tm}{\times}%
\newcommand{\End}{\operatorname{End}}%
\newcommand{\id}{\operatorname{id}}%
\newcommand{\dist}{\operatorname{dist}}%
\newcommand{\st}{\operatorname{st}}%
\newcommand{\ep}{\varepsilon}%
\newcommand{\MC}{\mathcal{M}}%
\newcommand{\RC}{\mathcal{R}}%
\newcommand{\rme}{\mathrm{e}}%
\newcommand{\rmS}{\mathrm{S}}%
\newcommand{\SO}{\mathrm{SO}}%

\begin{document}

\title{Topological Conjugacy of Real Projective Flows}% 
\author{V. Ayala\footnote{Departamento de Matem\'atica, Universidad Cat\'olica del Norte, Antofagasta, Chile, vayala@ucn.cl. This author was supported by Fondecyt Project no.~1100375.}, 
C. Kawan\footnote{Universit\"{a}t Augsburg, Universit\"{a}tsstrasse 14, 86159 Augsburg, Germany, christoph.kawan@math.uni-augsburg.de. This author was supported by DFG grant Co 124/17-2 within DFG priority program 1305.}}%
\maketitle%

\begin{center}
{\small\it We dedicate this paper to Fritz Colonius and Wolfgang Kliemann}%
\end{center}

\begin{abstract}
In this paper we prove the following topological classification result for flows on real projective space induced by linear flows on Euclidean space: Two flows on the projective space $\P(V)$ of a finite-dimensional real vector space $V$, induced by endomorphisms $A$ and $B$ of $V$, are topologically conjugate if and only if the Jordan structures of $A$ and $B$ coincide except for the real parts of the eigenvalues whose values may differ but whose
order and multiplicities must agree. Our proof is mainly based on ideas of Kuiper who considered the discrete-time analogue of this classification problem. We also correct a mistake in Kuiper's proof.%
\end{abstract}

\begin{center}
{\small{\bf Keywords:} Topological Conjugacy; Flows; Jordan Form; Projective Space}%
\end{center}

\section{Introduction}
\label{intro}

The topological classification of linear dynamical systems or, more general, linear group actions has a long history. In fact, it goes back to Poincar\'e \cite{Poi} who knew that orthogonal matrices in dimension $2$ are topologically conjugate if and only if they are linearly conjugate. De Rham \cite{DR1} conjectured that this equivalence is true in arbitrary finite dimensions. The first counterexamples to his conjecture were given by Cappell and Shaneson \cite{CS1}, who also proved that the conjecture holds up to dimension five. We refer to the articles \cite{CS1,CS2,CS3,Cea,DR1,DR2,HPe,HPa,KRo,Sch,Str} for more information about this thread of research. In particular, we recommend the introduction of Cruz \cite{Cru} for a more detailed historical account. There are two special cases of linear actions which are much easier to handle than the general case, namely, hyperbolic actions and continuous-time linear flows. For instance, Strelcyn \cite{Str} proved that two hyperbolic linear operators on a complex Banach space are topologically conjugate if and only if the dimensions of the stable and unstable subspaces coincide. Kuiper \cite{Ku1} and independently Ladis \cite{Lad} provided a complete classification of linear flows in finite dimensions. Also equivalence by H\"{o}lder or Lipschitz maps has been considered for linear flows and complete classification results are available, see \cite{KSt,MMe}.%

Another thread of research concerns actions on compact manifolds induced by linear actions. Naturally, every automorphism of a finite-dimensional vector space $V$ induces a diffeomorphism on the corresponding projective space, or more general, on the Grassmann manifold of $k$-dimensional subspaces of $V$, and the flag manifolds whose elements are the flags $V_1 \subset \cdots \subset V_r$ of linear subspaces $V_i \subset V$ of fixed dimensions. In Batterson \cite{Bat}, one finds a characterization of the structurally stable diffeomorphisms of this kind on the Grassmann manifolds. In Ayala, Colonius and Kliemann \cite{ACK}, the Lyapunov forms of matrices are characterized topologically by studying Morse decompositions of their induced flows on the Grassmann and flag manifolds.%

A special case of such induced systems has been treated by Kuiper \cite{Ku2}, who considered the discrete-time dynamical system on a real projective space induced by a linear automorphism. He provided an almost complete topological classification of such systems saying that two projective transformations induced by invertible matrices $A$ and $B$ are topologically conjugate if $A$ and $B$ can be written in the form%
\begin{eqnarray*}
  A &=& \lambda_1\sigma_1 \oplus \cdots \oplus \lambda_k\sigma_k,\\
  B &=& \mu_1\sigma_1 \oplus \cdots \oplus \mu_k\sigma_k,%
\end{eqnarray*}
where $\lambda_1>\cdots>\lambda_k>0$, $\mu_1>\mu_2>\cdots>\mu_k>0$, and each $\sigma_i$ is an automorphism all of whose eigenvalues have absolute value one, and the converse statement holds under some restrictive condition on the periods of the periodic points of the projective maps. The problem which leads to this restriction is directly related to the problem of the classification of linear transformations on Euclidean space described in the first paragraph. However, for projective flows no such restrictions are necessary, since the complete classification of linear flows by Kuiper and Ladis is available. Combining this classification result with the ideas of \cite{Ku2}, we prove the following classification result for projective flows: Two projective flows induced by linear flows $\rme^{At}$ and $\rme^{Bt}$, where $A$ and $B$ are endomorphisms of the finite-dimensional real vector space $V$, are topologically conjugate if and only if we can write (with respect to individual linear coordinates)%
\begin{eqnarray*}
  A &=& (\lambda_1 \id + \sigma_1) \oplus (\lambda_2\id + \sigma_2) \oplus \cdots \oplus (\lambda_k\id + \sigma_k),\\
  B &=& (\mu_1 \id + \sigma_1) \oplus (\mu_2\id + \sigma_2) \oplus \cdots \oplus (\mu_k\id + \sigma_k),%
\end{eqnarray*}
with real numbers $\lambda_1 > \lambda_2 > \cdots > \lambda_k$ and $\mu_1 > \mu_2 > \cdots > \mu_k$, and endomorphisms $\sigma_1,\ldots,\sigma_k$ all of whose eigenvalues are located on the imaginary axis.%

The paper is organized as follows. In Section \ref{sec_constr}, we prove the direction of the classification result, which involves the construction of a topological conjugacy. Here we follow the lines of Kuiper's proof and adapt his arguments to the continuous-time case. One of the main ideas of this proof consists in the definition of a function on projective space which increases along certain trajectories and allows to define fundamental domains for the corresponding flows. Then the fundamental domain method can be applied to construct the topological conjugacy. In Section \ref{sec_dynprops}, several dynamical invariants of the projective flows are described in algebraic terms in order to prove the missing direction of the classification result, namely the finest Morse decomposition, the recurrent set, and the dimensions of the stable manifolds. Here we correct a mistake in Kuiper's proof whose formulas for the dimensions of the stable manifolds (in the discrete-time case) are not correct. In the final Section \ref{sec_mt} we explain how the main result follows from the partial results of the preceding sections.%

\section{Preliminaries}\label{sec_prelim}%

Let $\phi_1:\R\tm X\rightarrow X$ and $\phi_2:\R\tm Y\rightarrow Y$ be continuous flows on topological spaces $X$ and $Y$. A homeomorphism $h:X\rightarrow Y$ is called a \emph{topological conjugacy} from $\phi_1$ to $\phi_2$ if%
\begin{equation*}
  h(\phi_1(t,x)) = \phi_2(t,h(x)) \mbox{\quad for all } t\in\R,\ x\in X.%
\end{equation*}
If such $h$ exists, we say that $\phi_1$ and $\phi_2$ are \emph{topologically conjugate}.%

By $\End(V)$ we denote the space of all endomorphisms of a finite-dimensional real vector space $V$. Every $A\in\End(V)$ induces a linear flow on $V$ by%
\begin{equation*}
  \varphi_A(t,x) = \rme^{At}x,\quad \varphi_A:\R\tm V\rightarrow V.%
\end{equation*}
By $\P(V)$ we denote the projective space of $V$, that is, the quotient space of $V^* := V\backslash\{0\}$ with respect to the equivalence relation $v\sim w$ if and only if $w = \alpha v$ for some nonzero $\alpha\in\R$. Hence, the elements of $\P(V)$ are the lines through the origin in $V$ (minus the origin itself). Since each time-$t$-map $\varphi_A(t,\cdot)$ maps such lines onto such lines, the flow $\varphi_A$ induces a flow $\psi_A$ on $\P(V)$ which we call the \emph{projective flow} associated with $A$. The natural projection $\P:V^* \rightarrow \P(V)$, $x \mapsto \P x := [x]_{\sim}$, is a continuous surjection which satisfies%
\begin{equation*}
  \psi_A(t,\P x) = \P \rme^{At}x \mbox{\quad for all } t\in\R,\ x\in V^*.%
\end{equation*}
If two projective flows $\psi_A$ and $\psi_B$ are topologically conjugate, we write $A \cong_{\P} B$.%

If $W$ is a linear subspace of the vector space $V$, then $\P W := \{\P x : x\in W^*\}$ is called a projective subspace of $\P(V)$. More general, we use the notation $\P A = \{\P x : x\in A\backslash \{0\}\}$ for any subset $A\subset V$.%

\section{Construction of Conjugacies}\label{sec_constr}%

In this section, we prove the following theorem:%

\begin{theorem}\label{thm_projconj}
Assume that $A,B\in\End(V)$ can be written in the form%
\begin{eqnarray*}
  A &=& (\lambda_1 \id + \sigma_1) \oplus (\lambda_2\id + \sigma_2) \oplus \cdots \oplus (\lambda_k\id + \sigma_k),\\
  B &=& (\mu_1 \id + \sigma_1) \oplus (\mu_2\id + \sigma_2) \oplus \cdots \oplus (\mu_k\id + \sigma_k),%
\end{eqnarray*}
with real numbers $\lambda_1 > \lambda_2 > \cdots > \lambda_k$, $\mu_1 > \mu_2 > \cdots > \mu_k$, and endomorphisms $\sigma_1,\ldots,\sigma_k$ with eigenvalues lying on the imaginary axis. Then $A \cong_{\P} B$.%
\end{theorem}

We will conclude this theorem from the following lemma:%

\begin{lemma}\label{lem_projconj}
Let $\lambda_1 > \lambda_2 > \cdots > \lambda_k$, $j\in\{1,\ldots,k\}$, and $\gamma\in\R$ with%
\begin{equation}\label{eq_gammacondition}
  \gamma + \lambda_j > \lambda_{j+1}.%
\end{equation}
Then there exists a topological conjugacy from the projective flow corresponding to the endomorphism%
\begin{equation*}
  A := (\lambda_1\id + \sigma_1) \oplus \cdots \oplus (\lambda_k\id + \sigma_k)%
\end{equation*}
to the projective flow corresponding to%
\begin{eqnarray*}
  B := ((\gamma + \lambda_1)\id + \sigma_1) &\oplus& \cdots \oplus ((\gamma + \lambda_j)\id + \sigma_j)\\
                                            &\oplus& (\lambda_{j+1}\id + \sigma_{j+1}) \oplus \cdots \oplus (\lambda_k\id + \sigma_k).%
\end{eqnarray*}
\end{lemma}

Indeed, assume that Lemma \ref{lem_projconj} is true. Then Theorem \ref{thm_projconj} is proved as follows.%

\begin{proof}[of Theorem~{\rm\ref{thm_projconj}}]
Define real numbers $\gamma_1,\ldots,\gamma_k$ recursively by%
\begin{eqnarray*}
  \mu_1 &=& \lambda_1 + \sum_{i=1}^k\gamma_i,\allowdisplaybreaks\\
  \mu_2 &=& \lambda_2 + \sum_{i=2}^k\gamma_i,\allowdisplaybreaks\\
        &\vdots& \allowdisplaybreaks\\
  \mu_{k-1} &=& \lambda_{k-1} + \gamma_{k-1} + \gamma_k,\allowdisplaybreaks\\
  \mu_k &=& \lambda_k + \gamma_k,%
\end{eqnarray*}
or briefly,%
\begin{equation*}
  \mu_j - \lambda_j = \sum_{i=j}^k\gamma_i,\qquad j=1,\ldots,k.%
\end{equation*}
Then Lemma \ref{lem_projconj} implies%
\begin{eqnarray*}
  (\lambda_1\id \!&+&\! \sigma_1) \oplus \cdots \oplus (\lambda_k\id + \sigma_k)\allowdisplaybreaks\\
  &\cong_{\P}& ((\gamma_k + \lambda_1)\id + \sigma_1) \oplus \cdots \oplus ((\gamma_k + \lambda_k)\id + \sigma_k)\allowdisplaybreaks\\
  &=& ((\gamma_k + \lambda_1)\id + \sigma_1) \oplus \cdots \oplus ((\gamma_k + \lambda_{k-1})\id + \sigma_{k-1}) \oplus (\mu_k\id + \sigma_k)\allowdisplaybreaks\\
  &\cong_{\P}& ((\gamma_{k-1}+\gamma_k + \lambda_1)\id + \sigma_1)\oplus \cdots \oplus((\gamma_{k-1}+\gamma_k+\lambda_{k-1})\id + \sigma_{k-1})\allowdisplaybreaks\\
  && \oplus (\mu_k\id + \sigma_k)\allowdisplaybreaks\\
  &=& ((\gamma_{k-1}+\gamma_k+\lambda_1)\id + \sigma_1) \oplus \cdots \oplus (\mu_{k-1}\id + \sigma_{k-1}) \oplus (\mu_k\id + \sigma_k)\allowdisplaybreaks\\
  &\vdots&\allowdisplaybreaks\\
  &\cong_{\P}& \left(\left(\sum_{i=1}^k\gamma_i + \lambda_1\right)\id + \sigma_1\right) \oplus (\mu_2\id + \sigma_2) \oplus \cdots \oplus (\mu_k\id + \sigma_k)\allowdisplaybreaks\\
  &=& (\mu_1\id + \sigma_1) \oplus \cdots \oplus (\mu_k\id + \sigma_k).%
\end{eqnarray*}
Note that condition \eqref{eq_gammacondition} is satisfied, which here reads%
\begin{equation*}
  \gamma_j + \left(\sum_{i=j+1}^k\gamma_i + \lambda_j\right) = \sum_{i=j}^k\gamma_i + \lambda_j = \mu_j > \mu_{j+1},%
\end{equation*}
and holds by assumption.%
\end{proof}

In order to prove Lemma \ref{lem_projconj}, we also use the following lemma.%

\begin{lemma}\label{lem_adapatednorm}
Let $A$ be as in Lemma \ref{lem_projconj} and $j\in\{1,\ldots,k\}$. Then for each $\delta>0$ there exists a norm $\|\cdot\|_A$ on $V$ which satisfies%
\begin{equation*}
  \|x\|_A^2 = \|x_1\|_A^2 + \cdots + \|x_k\|_A^2,%
\end{equation*}
where $x_1,\ldots,x_k$ are the components of $x$ with respect to the decomposition $\sigma_1\oplus\cdots\oplus\sigma_k$, and such that%
\begin{eqnarray*}
  \|\rme^{\sigma_i t}x_i\|_A &\geq& \rme^{-\delta t}\|x_i\|_A,\quad i=1,\ldots,j,\\
  \|\rme^{\sigma_i t}x_i\|_A &\leq& \rme^{\delta t}\|x_i\|_A,\quad i=j+1,\ldots,k,%
\end{eqnarray*}
for all $t\geq0$ and $x_i$ in the invariant subspace $V_i$ corresponding to $\sigma_i$.%
\end{lemma}

\begin{proof}
Let $\|\cdot\|$ be a fixed Euclidean norm on $V$. Since $\rme^{\sigma_it}$ has polynomial growth, for each $i\in\{j+1,\ldots,k\}$ there exists a constant $c_i=c_i(\delta)>0$ with%
\begin{equation*}
  \left\|\rme^{(\lambda_i\id + \sigma_i)t}x_i\right\| \leq c_i\rme^{(\lambda_i + \frac{2}{3}\delta)t}\|x_i\| \mbox{\quad for all\ } x_i\in V_i,\ t\geq0.%
\end{equation*}
Multiplication by $\rme^{(-\lambda_i-\frac{4}{3}\delta)t}$ gives%
\begin{equation*}
  \left\|\rme^{(-\frac{4}{3}\delta\id + \sigma_i)t}x_i\right\| \leq c_i\rme^{-\frac{2}{3}\delta t}\|x_i\| \mbox{\quad for all\ } x_i\in V_i,\ t\geq0.%
\end{equation*}
There exists an adapted norm $\|\cdot\|_{A,i}$ on $V_i$ such that (cf.~Robinson \cite{Rob})%
\begin{equation*}
  \left\|\rme^{(-\frac{4}{3}\delta\id + \sigma_i)t}x_i\right\|_{A,i} \leq \rme^{-\frac{1}{3}\delta t}\|x_i\|_{A,i} \mbox{\quad for all\ } x_i\in V_i,\ t\geq0.%
\end{equation*}
Multiplication by $\rme^{\frac{4}{3}\delta t}$ now yields%
\begin{equation*}
  \left\|\rme^{\sigma_it}x_i\right\|_{A,i} \leq \rme^{\delta t}\|x_i\|_{A,i}  \mbox{\quad for all\ } x_i\in V_i,\ t\geq0.%
\end{equation*}
Analogously, for $i\in\{1,\ldots,j\}$ there is a constant $c_i>0$ with%
\begin{equation*}
  \left\|\rme^{-(\lambda_i\id + \sigma_i)t}x_i\right\| \leq c_i\rme^{(-\lambda_i + \frac{2}{3}\delta)t}\|x_i\| \mbox{\quad for all\ } x_i\in V_i,\ t\geq0.%
\end{equation*}
Multiplication by $\rme^{(\lambda_i-\frac{4}{3}\delta)t}$ yields%
\begin{equation*}
  \left\|\rme^{(-\frac{4}{3}\delta\id - \sigma_i)t}x_i\right\| \leq c_i\rme^{-\frac{2}{3}\delta t}\|x_i\| \mbox{\quad for all\ } x_i\in V_i,\ t\geq0.%
\end{equation*}
Again, there is an adapted norm $\|\cdot\|_{A,i}$ on $V_i$ satisfying%
\begin{equation*}
  \left\|\rme^{(-\frac{4}{3}\delta\id - \sigma_i)t}x_i\right\|_{A,i} \leq \rme^{-\frac{1}{3}\delta t}\|x_i\|_{A,i} \mbox{\quad for all\ } x_i\in V_i,\ t\geq0.%
\end{equation*}
Multiplication by $\rme^{\frac{4}{3}\delta t}$ and replacing $x_i$ by $\rme^{\sigma_i t}x_i$ leads to%
\begin{equation*}
  \|\rme^{\sigma_it}x_i\|_{A,i} \geq \rme^{-\delta t}\|x_i\|_{A,i} \mbox{\quad for all\ } x_i\in V_i,\ t\geq0.%
\end{equation*}
Now it is easy to see that the desired norm is given by%
\begin{equation*}
  \|x\|_A := \left(\sum_{i=1}^k\|x_i\|_{A,i}^2\right)^{1/2} \mbox{\quad for all\ } x\in V.%
\end{equation*}
This finishes the proof.%
\end{proof}

\begin{proof}[of Lemma~{\rm\ref{lem_projconj}}]
The proof proceeds in four steps.%

\emph{Step 1.} We make some definitions: Let $x_1,\ldots,x_k$ denote the components of a vector $x\in V^*$ with respect to the decomposition $\sigma_1\oplus\cdots\oplus\sigma_k$. Let $\delta>0$ be chosen small enough such that%
\begin{equation}\label{eq_deltacondition}
  \max\left\{\lambda_{j+1}-\lambda_j+2\delta,\lambda_{j+1}-(\gamma+\lambda_j)+2\delta\right\} < 0.%
\end{equation}
This is possible, since $\lambda_j>\lambda_{j+1}$ and $\gamma + \lambda_j>\lambda_{j+1}$. For the chosen $\delta$, let $\|\cdot\|_A$ be a corresponding adapted norm as in Lemma \ref{lem_adapatednorm}, and define a function%
\begin{equation*}
  \alpha:\P(V) \rightarrow [0,\infty],\qquad p = \P x \mapsto \frac{\sum_{i=j+1}^k\|x_i\|_A^2}{\sum_{i=1}^j\|x_i\|_A^2}.%
\end{equation*}
It is easy to see that $\alpha$ is well-defined and continuous. We further define%
\begin{equation*}
  \beta:\P(V) \rightarrow [-\infty,\infty],\qquad p \mapsto \ln\alpha(p).%
\end{equation*}
Consider the complementary linear subspaces%
\begin{eqnarray*}
  W &:=& \left\{x\in V\ :\ x_{j+1} = \ldots = x_k = 0\right\},\\
  Z &:=& \left\{x\in V\ :\ x_1 = \ldots = x_j = 0\right\},% 
\end{eqnarray*}
and the corresponding projective subspaces $\P W$ and $\P Z$. Obviously,%
\begin{equation*}
  \P W = \beta^{-1}(-\infty) \mbox{\quad and\quad } \P Z = \beta^{-1}(\infty).%
\end{equation*}
For brevity in notation we write%
\begin{equation*}
  \P(V)^* := \P(V) \backslash (\P W \cup \P Z).%
\end{equation*}

\emph{Step 2.} We analyze the behavior of the projective flows $\psi_A$ and $\psi_B$ on $\P(V)^*$: Take $x\in V\backslash (W\cup Z)$, that is, $x = x_W \oplus x_Z$ with $x_W \in W^*$ and $x_Z \in Z^*$, and let $t\geq0$. Then%
\begin{eqnarray*}
  \alpha(\P\rme^{At}x) &=& \alpha\left(\P(\rme^{At}x_W + \rme^{At}x_Z)\right) = \frac{\|\rme^{At}x_Z\|_A^2}{\|\rme^{At}x_W\|_A^2}\\
  &=& \frac{\sum_{i=j+1}^k\rme^{2\lambda_it}\|\rme^{\sigma_it}x_i\|_A^2}{\sum_{i=1}^j\rme^{2\lambda_it}\|\rme^{\sigma_it}x_i\|_A^2}\\
  &\leq& \frac{\rme^{2\lambda_{j+1}t}\sum_{i=j+1}^k\rme^{2(\lambda_i-\lambda_{j+1})t}\rme^{2\delta t}\|x_i\|_A^2}{\rme^{2\lambda_jt}\sum_{i=1}^j\rme^{2(\lambda_i-\lambda_j)t}\rme^{-2\delta t}\|x_i\|_A^2}\\
  &=& \rme^{2(\lambda_{j+1}-\lambda_j + 2\delta)t}\frac{\sum_{i=j+1}^k \rme^{2(\lambda_i-\lambda_{j+1})t}\|x_i\|_A^2}{\sum_{i=1}^j\rme^{2(\lambda_i-\lambda_j)t}\|x_i\|_A^2}.%
\end{eqnarray*}
Since $\lambda_i - \lambda_{j+1}\leq 0$ for $i\geq j+1$ and $\lambda_i-\lambda_j\geq0$ for $i\leq j$, we obtain%
\begin{equation}\label{eq_alphaineq}
  \alpha(\P\rme^{At}x) \leq \rme^{2(\lambda_{j+1}-\lambda_j+2\delta)t}\frac{\sum_{i=j+1}^k\|x_i\|_A^2}{\sum_{i=1}^j\|x_i\|_A^2} = \rme^{2(\lambda_{j+1}-\lambda_j+2\delta)t}\alpha(\P x),%
\end{equation}
or equivalently,%
\begin{equation}\label{eq_betadecrftime}
  \beta(\P\rme^{At}x) \leq \beta(\P x) + 2(\lambda_{j+1} - \lambda_j + 2\delta)t \mbox{\quad for all } t\geq0.%
\end{equation}
This inequality holds for all $x\in V\backslash(W\cup Z)$ and hence we can replace $x$ by $\rme^{-At}x$, which yields%
\begin{equation}\label{eq_betadecrbtime}
  \beta(\P\rme^{At}x) \geq \beta(\P x) + 2(\lambda_{j+1} - \lambda_j + 2\delta)t \mbox{\quad for all } t\leq 0.%
\end{equation}
By \eqref{eq_deltacondition} we have $\lambda_{j+1}-\lambda_j+2\delta<0$. Hence, $\beta$ is strictly decreasing along the trajectory through $\P x$ and \eqref{eq_betadecrftime}, \eqref{eq_betadecrbtime} imply%
\begin{equation*}
  \beta(\P\rme^{At}x) \rightarrow \mp\infty \mbox{\quad for } t\rightarrow \pm \infty.%
\end{equation*}
Analogously, for $B$ one shows that%
\begin{equation}\label{eq_betadecrftime_B}
  \beta(\P\rme^{Bt}x) \leq \beta(\P x) + 2(\lambda_{j+1} - (\gamma+\lambda_j) + 2\delta)t \mbox{\quad for all\ } t\geq 0,%
\end{equation}
and from \eqref{eq_deltacondition} it follows that the trajectories of $\psi_B|_{\P(V)^*}$ have the same limit behavior as those of $\psi_A|_{\P(V)^*}$. In forward time they converge to $\P W$ and in backward time to $\P Z$.%

\emph{Step 3.} Using the fundamental domain method, we show that $\psi_A$ and $\psi_B$ are topologically conjugate on $\P(V)^*$: We can define a fundamental domain for both $\psi_A|_{\P(V)^*}$ and $\psi_B|_{\P(V)^*}$ by $D := \beta^{-1}(0)$. Then every trajectory of $\psi_A|_{\P(V)^*}$ and $\psi_B|_{\P(V)^*}$ intersects $D$ in exactly one point, which follows from \eqref{eq_betadecrftime} and \eqref{eq_betadecrftime_B}. Define a topological conjugacy by%
\begin{equation*}
  h:\P(V)^* \rightarrow \P(V)^*,\qquad h(\psi_A(t,p)) := \psi_B(t,p) \mbox{\quad for all\ } t\in\R,\ p\in D.%
\end{equation*}
Since $D$ is a fundamental domain, $h$ is well-defined. Obviously, $h$ is invertible; its inverse maps $\psi_B(t,p)$ to $\psi_A(t,p)$ for $t\in\R$ and $p\in D$. An explicit expression for $h$ is%
\begin{equation}\label{eq_hexplicit}
  h(p) = \psi_B(-\tau(p),\psi_A(\tau(p),p)),%
\end{equation}
where $\tau:\P(V)^* \rightarrow \R$ is defined implicitly by the equation%
\begin{equation*}
  \beta(\psi_A(\tau(p),p)) = 0.%
\end{equation*}
To show that $h$ is continuous, it hence suffices to prove continuity of $\tau$. To this end, assume that $\tau$ is not continuous at $p\in\P(V)^*$. Then there exist $\ep>0$ and a sequence $(p_n)_{n\geq1}$ converging to $p$ such that $|\tau(p_n)-\tau(p)|\geq\ep$ for all $n\geq1$. We can choose a subsequence $(p_{m_n})$ such that $\tau(p_{m_n})$ either converges to a real number $\tau^*$, to $\infty$, or to $-\infty$. For $\tau(p_{m_n})\rightarrow\pm\infty$ we would have $\psi_A(\tau(p_{m_n}),p_{m_n})\rightarrow \P W \cup \P Z$. This is not possible, since $\beta(\psi_A(\tau(p_{m_n}),p_{m_n}))=0$ for all $n$. If $\tau(p_{m_n}) \rightarrow \tau^*\in\R$, then%
\begin{equation*}
  0 = \lim_{n\rightarrow\infty}\beta(\psi_A(\tau(p_{m_n}),p_{m_n})) = \beta(\psi_A(\tau^*,p)) \quad \Rightarrow \quad \tau^* = \tau(p),%
\end{equation*}
in contradiction to $|\tau(p_{m_n})-\tau(p)|\geq\ep$. The conjugacy identity easily follows from the definition of $h$.%

\emph{Step 4.} We show that $h$ can be extended to a conjugacy on $\P(V)$. To this end, we define%
\begin{equation*}
  \overline{h}(p) := \left\{\begin{array}{cc}
                                  h(p) & \mbox{for } p\in\P(V)^*,\\
                                    p  & \mbox{for } p\in \P W \cup \P Z.%
                            \end{array}\right.%
\end{equation*}
Then $\overline{h}$ is bijective, continuous on $\P(V)^*$ and on $\P W \cup \P Z$. Moreover, it satisfies the conjugacy identity, since on $\P W \cup \P Z$ the flows $\psi_A$ and $\psi_B$ coincide. It remains to prove that $\overline{h}$ is continuous on $\P(V)$. To this end, consider the explicit expression \eqref{eq_hexplicit} for $h$ and pick $x = x_W \oplus x_Z \in V\backslash (W\cup Z)$. Then%
\begin{eqnarray*}
  h(\P x) &=& \psi_B(-\tau(\P x),\psi_A(\tau(\P x),\P x))\\
          &=& \psi_B\left(-\tau(\P x),\P\rme^{\tau(\P x)A}(x_W \oplus x_Z)\right)\\
          &=& \P\underbrace{\rme^{-\tau(\P x)B}\rme^{\tau(\P x)A}}_{=:C(x)}(x_W \oplus x_Z) = \P(C(x)x_W \oplus C(x)x_Z).%
\end{eqnarray*}
Since $A$ and $B$ commute, we have $C(x) = \rme^{\tau(\P x)(A-B)}$ and hence%
\begin{equation}\label{eq_cxdescr}
  C(x) = \rme^{-\gamma \tau(\P x)}\id_W \oplus \id_Z \quad \Rightarrow \quad h(\P x) = \P(\rme^{-\gamma\tau(\P x)}x_W \oplus x_Z).%
\end{equation}
Now let $(p_n)_{n\geq1}$ be a sequence in $\P(V)^*$ converging to some $p\in\P W$ (without loss of generality). To show that $h(p_n)\rightarrow p = \overline{h}(p)$, we write%
\begin{equation*}
  p_n = \P\rme^{At_n}x_n,\quad t_n\in\R,\ \P x_n \in D = \beta^{-1}(0)%
\end{equation*}
and $x_n = x_{W,n} \oplus x_{Z,n}$ with $x_{W,n}\in W^*$, $x_{Z,n}\in Z^*$. Obviously, we may assume that $t_n>0$. Then we have $t_n\rightarrow\infty$, which is proved as follows: We have $\alpha(p_n)=\rme^{\beta(p_n)} \rightarrow 0$ and%
\begin{eqnarray*}
  \alpha(p_n) &=& \frac{\|\rme^{At_n}x_{Z,n}\|_A^2}{\|\rme^{At_n}x_{W,n}\|_A^2} \geq \frac{(c_1\rme^{(\lambda_k-\ep)t_n})^2\|x_{Z,n}\|_A^2}{(c_2\rme^{(\lambda_1+\ep)t_n})^2\|x_{W,n}\|_A^2}\\
  &=& \left[\frac{c_1}{c_2}\rme^{(\lambda_k-\lambda_1-2\ep)t_n}\right]^2 \underbrace{\alpha(\P x_n)}_{=1}%
\end{eqnarray*}
with small $\ep>0$ and constants $c_1,c_2>0$, which implies $(\lambda_k-\lambda_1-2\ep)t_n \rightarrow -\infty$ and therefore $t_n\rightarrow\infty$. We also have%
\begin{equation*}
  \tau(p_n) = \tau(\P\rme^{At_n}x_n) = -t_n,%
\end{equation*}
and hence \eqref{eq_cxdescr} yields%
\begin{eqnarray*}
  h(p_n) &=& \P\left(\rme^{\gamma t_n}\rme^{At_n}x_{W,n} \oplus \rme^{At_n}x_{Z,n}\right)\\
         &=& \P\left(\frac{\rme^{At_n}x_{W,n}}{\|\rme^{At_n}x_n\|_A} + \frac{\rme^{At_n}x_{Z,n}}{\rme^{\gamma t_n}\|\rme^{At_n}x_n\|_A}\right)\\
         &=& \P\left(\frac{\rme^{At_n}x_n}{\|\rme^{At_n}x_n\|_A} + \left[\frac{\rme^{At_n}x_{Z,n}}{\rme^{\gamma t_n}\|\rme^{At_n}x_n\|_A}-\frac{\rme^{At_n}x_{Z,n}}{\|\rme^{At_n}x_n\|_A}\right]\right).%
\end{eqnarray*}
Observe that%
\begin{eqnarray*}
  \left\|\frac{\rme^{At_n}x_{Z,n}}{\|\rme^{At_n}x_n\|_A}\right\|_A^2 &=& \frac{\|\rme^{At_n}x_{Z,n}\|_A^2}{\|\rme^{At_n}x_n\|_A^2} = \frac{\|\rme^{At_n}x_{Z,n}\|_A^2}{\|\rme^{At_n}x_{W,n}\|_A^2 + \|\rme^{At_n}x_{Z,n}\|_A^2}\\
  &=& \left(\alpha(p_n)^{-1} + 1\right)^{-1} = \frac{\alpha(p_n)}{1+\alpha(p_n)} \rightarrow 0.%
\end{eqnarray*}
Moreover, we have%
\begin{equation*}
  \left\|\frac{\rme^{At_n}x_{Z,n}}{\rme^{\gamma t_n}\|\rme^{At_n}x_n\|_A}\right\|_A^2 = \frac{\rme^{-2\gamma t_n}\alpha(p_n)}{1+\alpha(p_n)} \rightarrow 0,%
\end{equation*}
since by \eqref{eq_alphaineq}%
\begin{eqnarray*}
  \rme^{-2\gamma t_n}\alpha(p_n) &=& \rme^{-2\gamma t_n}\alpha(\P\rme^{At_n}x_n)\\
  &\leq& \rme^{-2\gamma t_n}\rme^{2(\lambda_{j+1}-\lambda_j+2\delta)t_n}\underbrace{\alpha(\P x_n)}_{=1}\\
  &=& \rme^{2(\lambda_{j+1}-(\gamma+\lambda_j)+2\delta)t_n} \rightarrow 0.%
\end{eqnarray*} 
This gives%
\begin{equation*}
  \lim_{n\rightarrow\infty} h(p_n) = \lim_{n\rightarrow\infty} \P\left(\frac{\rme^{At_n}x_n}{\|\rme^{At_n}x_n\|_A}\right) = p,%
\end{equation*}
which concludes the proof.%
\end{proof}

\section{Dynamical Invariants of Projective Flows}\label{sec_dynprops}%

In this section, we describe several dynamical invariants of projective flows in algebraic terms.%

\subsection{The Finest Morse Decomposition}%

We start by giving a description of the finest Morse decomposition of a projective flow. Let us first recall some notions.%

\begin{definition}
Let $\varphi:\R\tm X\rightarrow X$ be a continuous flow on a metric space $X$. The \emph{$\alpha$-limit set} and the \emph{$\omega$-limit set} of a point $x\in X$ are defined by%
\begin{eqnarray*}
  \alpha(x,\varphi) &:=& \left\{y\in X\ |\ \exists t_n \rightarrow -\infty:\ \varphi(t_n,x) \rightarrow y\right\},\\
  \omega(x,\varphi) &:=& \left\{y\in X\ |\ \exists t_n \rightarrow \infty:\ \varphi(t_n,x) \rightarrow y\right\}.%
\end{eqnarray*}
\end{definition}

\begin{definition}
Let $\varphi:\R\tm X\rightarrow X$ be a continuous flow on a compact metric space $X$. A compact set $K\subset X$ is called \emph{isolated invariant} if it is invariant (that is, $\varphi^t(K)\subset K$ for all $t\in\R$) and if there is a neighborhood $N$ of $K$ such that the implication%
\begin{equation*}
  \varphi(t,x) \in N \mbox{\quad for all } t\in\R \qquad \Rightarrow \qquad x \in K%
\end{equation*}
holds. A \emph{Morse decomposition} is a finite collection $\{\MC_i\}_{i=1}^n$ of nonempty, pairwise disjoint, and compact isolated invariant sets with the following properties:%
\begin{enumerate}
\item[(i)] For all $x\in X$ it holds that $\alpha(x,\varphi),\omega(x,\varphi)\subset \bigcup_{i=1}^n \MC_i$.%
\item[(ii)] Suppose there are $\MC_{j_0},\MC_{j_1},\ldots,\MC_{j_l}$ and $x_1,\ldots,x_l\in X\backslash \bigcup_{i=1}^n\MC_i$ with $\alpha(x_i,\varphi)\subset \MC_{j_{i-1}}$ and $\omega(x_i,\varphi)\subset \MC_{j_i}$ for $i=1,\ldots,l$. Then $\MC_{j_0}\neq \MC_{j_l}$.%
\end{enumerate}
The elements of a Morse decomposition are called \emph{Morse sets}. One can define an order on the Morse sets by%
\begin{equation*}
  \MC_i \preceq \MC_j \ :\Leftrightarrow \ \exists x\in X:\ \alpha(x,\varphi)\subset \MC_i \mbox{ and } \omega(x,\varphi)\subset \MC_j.%
\end{equation*}
A Morse decomposition is \emph{finer} than another one if the elements of the first one are contained in those of the second one. A finest Morse decomposition is a Morse decomposition which is finer than every other one.%
\end{definition}

It is obvious that a finest Morse decomposition, if it exists, is unique. A finest Morse decomposition exists if and only if the chain recurrent set of the given flow has only finitely many components. In this case, the Morse sets coincide with the chain recurrent components.%

For a projective flow, we have the following result.%

\begin{proposition}
Let $A\in\End(V)$ with associated projective flow $\psi_A$ on $\P(V)$. Then the components of the chain recurrent set of $\psi_A$ are the projective subspaces $\P V_1,\ldots,\P V_k$, where $V = V_1 \oplus \cdots \oplus V_k$ is the decomposition of $V$ into the Lyapunov spaces of $A$, that is, the sums of generalized eigenspaces corresponding to eigenvalues with the same real part. Consequently, $\{\P V_1,\ldots,\P V_k\}$ is the finest Morse decomposition of $\psi_A$. If we assume that real parts of the eigenvalues are $\lambda_1>\cdots>\lambda_k$ and $V_i$ corresponds to $\lambda_i$, then the order of the Morse sets is $\P V_k \preceq \P V_{k-1} \preceq \cdots \preceq \P V_1$.
\end{proposition}

A proof of this result can be found, for instance, in Ferraiol, Patr\~ao and Seco \cite[Section 4]{FPS} or Colonius and Kliemann \cite{CKl}. It is easy to see that a topological conjugacy from a flow $\phi_1$ to another flow $\phi_2$ maps a finest Morse decomposition of $\phi_1$ to a finest Morse decomposition of $\phi_2$ preserving the order of the Morse sets. Applying this to projective flows, we obtain the following result.%

\begin{proposition}\label{prop_morse}
Let $\psi_A$ and $\psi_B$ be two projective flows on $\P(V)$ such that there exists a topological conjugacy $h:\P(V)\rightarrow\P(V)$ from $\psi_A$ to $\psi_B$. Then, up to linear conjugacy, $A$ and $B$ can be written as%
\begin{eqnarray*}
  A &=& (\lambda_1\id + \sigma_1(A)) \oplus (\lambda_2\id + \sigma_2(A)) \oplus \cdots \oplus (\lambda_k\id + \sigma_k(A)),\\
  B &=& (\mu_1\id + \sigma_1(B)) \oplus (\mu_2\id + \sigma_2(B)) \oplus \cdots \oplus (\mu_k\id + \sigma_k(B)),%
\end{eqnarray*}
where $\lambda_1>\cdots>\lambda_k$, $\mu_1>\cdots>\mu_k$, and $\sigma_i(A)$, $\sigma_i(B)$ are endomorphisms of the same dimension with eigenvalues lying on the imaginary axis.%
\end{proposition}

\subsection{The Recurrent Set}%

Let $V$ be a Euclidean space of dimension $n+1$ with inner product $\langle\cdot,\cdot\rangle$ and associated norm $\|\cdot\|$. Then a linear flow $\varphi_A(t,x) = \rme^{At}x$ on $V$ induces a flow on the $n$-dimensional unit sphere $\rmS(V) := \{x\in V\ :\ \|x\|=1\}$ given by%
\begin{equation*}
  \chi_A(t,x) := \frac{\rme^{At}x}{\|\rme^{At}x\|}.%
\end{equation*}
Writing $\pi := \P|_{\rmS(V)}:\rmS(V) \rightarrow \P(V)$, we find that%
\begin{equation*}
  \pi\chi_A(t,x) = \psi_A(t,\pi x) \mbox{\quad for all\ } t\in\R,\ x\in\rmS(V).%
\end{equation*}
The following lemma shows that a conjugacy between two projective flows can be lifted to a conjugacy between the corresponding flows on the sphere.%

\begin{lemma}\label{lem_conjugacylifting}
Let $A,B\in\End(V)$ and assume that there exists a topological conjugacy $h:\P(V)\rightarrow\P(V)$ from $\psi_A$ to $\psi_B$. Then there also exists a topological conjugacy $H:\rmS(V)\rightarrow\rmS(V)$ from $\chi_A$ to $\chi_B$.%
\end{lemma}

\begin{proof}
From covering theory it follows that $h$ can be lifted to a continuous map $H:\rmS(V)\rightarrow \rmS(V)$, that is, $\pi \circ H = h \circ \pi$ (cf., for instance, Hatcher \cite{Hat}). In fact, $H$ is a homeomorphism, since also $h^{-1}$ can be lifted to a continuous map $G:\rmS(V)\rightarrow\rmS(V)$. Therefore, $\pi = \pi \circ (G\circ H) = \pi \circ (H\circ G)$, which implies that $G\circ H$ and $H\circ G$ are deck transformations of the two-fold covering map $\pi$, hence $G\circ H,H\circ G \in \{\id,-\id\}$. This implies both injectivity and surjectivity of $H$. By compactness of $\rmS(V)$, $H$ must be a homeomorphism. Moreover,%
\begin{eqnarray*}
  \pi \circ H \circ \chi_A^t &=& h\circ\pi\circ\chi_A^t = h \circ \psi_A^t \circ \pi = \psi_B^t \circ h\circ \pi\\
                             &=& \psi_B^t \circ \pi \circ H = \pi \circ \chi_B^t \circ H.%
\end{eqnarray*}
Hence, for every $t\in\R$ and $x\in\rmS(V)$ we have $H(\chi_A^t(x))=\pm\chi_B^t(H(x))$. By continuity of both sides with respect to $t$, we see by putting $t=0$ that in fact $H(\chi_A^t(x))\equiv\chi_B^t(H(x))$. This concludes the proof.%
\end{proof}

\begin{remark}
The idea of lifting the conjugacy on projective space to the unit sphere is also used by Kuiper \cite{Ku2} in the discrete-time case. But in that case a problem remains concerning the signs of the lifted transformations, that is, if $A$ and $B$ are the corresponding isomorphisms of $V$ and $A^{\rmS}$ and $B^{\rmS}$ the associated maps on $\rmS(V)$, and the induced maps on $\P(V)$ are topologically conjugate, then the lifting argument shows that $A^{\rmS}$ is conjugate to either $B^{\rmS}$ or $-B^{\rmS}$, but it is not easy to see that it is $B^{\rmS}$.%
\end{remark}

We want to determine the recurrent set of $\chi_A$. First we recall the definitions of recurrent points and the recurrent set.%

\begin{definition}
Let $\varphi:\R\tm X\rightarrow X$ be a continuous flow on a compact metric space $X$. A point $x\in X$ is called \emph{recurrent} if $x\in\omega(x,\varphi)$. The set of all recurrent points of $\varphi$ is denoted by $\RC(\varphi)$.%
\end{definition}

Note that $\RC(\varphi)$ is invariant, but in general not closed in $X$. It is easy to see that a topological conjugacy from $\varphi_1$ to $\varphi_2$ maps $\RC(\varphi_1)$ onto $\RC(\varphi_2)$.%

\begin{lemma}\label{lem_skewsymmetricmatrix}
Let $S\in\End(V)$ be skew-symmetric with respect to the inner product on $V$. Then there exists a sequence $(t_n)_{n\geq1}$ of positive real numbers with $t_n\rightarrow\infty$ and $\rme^{St_n}\rightarrow\id_V$.%
\end{lemma}

\begin{proof}
For each $t\in\R$, $\rme^{St}$ is an element of the compact Lie group $\SO(V) = \{A\in\End(V) : AA^* = \id_V, \det A = 1\}$, where $A^*$ denotes the adjoint of $A$ with respect to the given inner product. Consider a left-invariant Riemannian metric on $\SO(V)$ and denote by $d$ the associated distance function. Consider a sequence $(\tau_n)_{n\geq1}$, $\tau_n>0$, $\tau_n\rightarrow\infty$, such that $\rme^{S\tau_n}$ converges. Then, for given $\ep>0$ we find $m=m(\ep)\in\N$ such that for all $l,k\geq m$ we have $d(\rme^{S(\tau_l-\tau_k)},\id_V) = d(\rme^{S\tau_l},\rme^{S\tau_k})<\ep$. Now for each $n\in\N$ let $\ep_n := 1/n$ and $m=m(\ep_n)$. Then choose $l>m(\ep_n)$ such that $t_n := \tau_l - \tau_{m(\ep_n)} > n$. This implies $d(\rme^{St_n},\id_V) < 1/n$ and hence proves the lemma.%
\end{proof}

Using the preceding lemma we can now characterize the recurrent set of $\chi_A$.%

\begin{proposition}\label{prop_recurrentsetsphere}
Let $A\in\End(V)$ be an endomorphism of $V$ all of whose eigenvalues lie on the imaginary axis, which is given in Jordan normal form with respect to an orthonormal basis of $V$. Let $E_A$ be the sum of the real eigenspaces of $A$ (that is, the subspaces of the form $V \cap (E(i\alpha)\oplus E(-i\alpha))$, where $E(\cdot)$ denotes the corresponding complex eigenspace). Then $\RC(\chi_A) = E_A \cap \rmS(V)$.%
\end{proposition}

\begin{proof}
The restriction $A|_{E_A}$ is skew-symmetric. By Lemma \ref{lem_skewsymmetricmatrix}, there is a sequence $t_n\rightarrow\infty$ with $\rme^{A|_{E_A}t_n}\rightarrow\id_{E_A}$, which implies%
\begin{equation*}
  \chi_A(t_n,x) = \frac{\rme^{At_n}x}{\|\rme^{At_n}x\|} = \rme^{At_n}x \rightarrow x \mbox{\quad for all\ } x\in E_A\cap\rmS(V).%
\end{equation*}
Therefore, $E_A\cap\rmS(V) \subset \RC(\chi_A)$. On the other hand, all $x\in\rmS(V)\backslash E_A$ are not recurrent, which is proved as follows. We can write%
\begin{equation*}
  x = x_1 \oplus x_2,\qquad x_1\in E_A,\ 0 \neq x_2 \in E_A^{\bot},%
\end{equation*}
and $A = S + N$, $SN=NS$, with a skew-symmetric endomorphism $S$ and a nilpotent endomorphism $N$ (the Jordan decomposition of $A$). Then $\rme^{St}$ is orthogonal and $\rme^{Nt}x_1=x_1$. This implies%
\begin{eqnarray*}
  \chi_A(t,x) &=& \frac{\rme^{At}x}{\|\rme^{At}x\|} = \frac{\rme^{(S+N)t}x}{\|\rme^{(S+N)t}x\|} = \frac{\rme^{St}\rme^{Nt}x}{\|\rme^{St}\rme^{Nt}x\|}\\
              &=& \frac{\rme^{St}\rme^{Nt}(x_1\oplus x_2)}{\|\rme^{Nt}(x_1\oplus x_2)\|} = \rme^{St}\frac{x_1\oplus \rme^{Nt}x_2}{\|x_1\oplus \rme^{Nt}x_2\|}\\
              &=& \rme^{St}\left[\frac{x_1}{\|x_1 \oplus \rme^{Nt}x_2\|} + \frac{\rme^{Nt}x_2}{\|x_1\oplus\rme^{Nt}x_2\|}\right].%
\end{eqnarray*}
Since $x_2\neq0$ and $x_2\notin E_A$, we have $\|\rme^{Nt}x_2\|\rightarrow\infty$. Hence,%
\begin{equation*}
  \left\|\chi_A(t,x) - \rme^{St}\frac{\rme^{Nt}x_2}{\|\rme^{Nt}x_2\|}\right\| \rightarrow 0 \mbox{\quad for\ } t\rightarrow\infty.%
\end{equation*}
Since $N$ is nilpotent, we have $N^jx_2\in E_A$ and $N^{j+1}x_2 = 0$ for some $j\in\{1,\ldots,n-1\}$. This implies%
\begin{eqnarray*}
  \frac{\rme^{Nt}x_2}{\|\rme^{Nt}x_2\|} &=& \frac{x_2 + tNx_2 + \frac{t^2}{2!}N^2x_2 + \cdots + \frac{t^j}{j!}N^jx_2}{\|x_2 + tNx_2 + \frac{t^2}{2!}N^2x_2 + \cdots + \frac{t^j}{j!}N^jx_2\|}\\ 
  &=& \frac{t^{-j}x_2 + t^{1-j}Nx_2 + \frac{t^{2-j}}{2!}N^2x_2 + \cdots + \frac{1}{j!}N^jx_2}{\|t^{-j}x_2 + t^{1-j}Nx_2 + \frac{t^{2-j}}{2!}N^2x_2 + \cdots + \frac{1}{j!}N^jx_2\|}\\
  &\xrightarrow{t\rightarrow\infty}& \frac{N^jx_2}{\|N^jx_2\|} \in E_A \cap \rmS(V).%
\end{eqnarray*}
Therefore, $x$ cannot be recurrent, which proves that $\RC(\chi_A) \subset E_A \cap \rmS(V)$.%
\end{proof}

Using the preceding proposition we can immediately determine the recurrent set of the projective flow $\psi_A$.%

\begin{corollary}
Under the assumptions of Proposition \ref{prop_recurrentsetsphere}, the recurrent set of $\psi_A$ is given by $\RC(\psi_A)=\P E_A$.%
\end{corollary}

\begin{proof}
We have $\P \circ \chi_A^t = \psi_A^t \circ \P$ for all $t\in\R$. Hence,%
\begin{equation*}
  \P E_A = \P(E_A \cap \rmS(V)) = \P(\RC(\chi_A)) \subset \RC(\psi_A).%
\end{equation*}
Now assume that $\P x\notin \P E_A$ for some $x\in\rmS(V)$. Then $x\notin E_A$ and hence the proof of Proposition \ref{prop_recurrentsetsphere} shows that $\dist(\chi_A(t,x),E_A\cap\rmS(V))\rightarrow0$ implying that $\psi_A(t,\P x) = \P\chi_A(t,x) \rightarrow \P E_A$. Therefore, $\P x \notin \RC(\psi_A)$.%
\end{proof}

\begin{remark}
Note that the preceding corollary also holds without the assumption that $A$ is given in Jordan normal form with respect to an orthonormal basis, since for any given $A$ one can choose an inner product such that this assumption holds, and the set $E_A$ does not depend on the inner product.%
\end{remark}

\begin{remark}
A characterization of the recurrent set of a projective flow can also be found in Ferraiol, Patr\~ao, Seco \cite{FPS}, and in Kuiper \cite{Ku2} for the discrete-time case.%
\end{remark}

From the algebraic description of the recurrent set on the unit sphere and the Kuiper-Ladis characterization of topological conjugacy for linear flows on Euclidean space, mentioned in the introduction, we can now conclude the following result.%

\begin{corollary}\label{cor_eigenvalues}
Let $A,B\in\End(V)$ be endomorphisms all of whose eigenvalues lie on the imaginary axis. If $\psi_A$ and $\psi_B$ are topologically conjugate, then the following assertions hold:%
\begin{enumerate}
\item[(a)] The spectra of $A$ and $B$ coincide.%
\item[(b)] The geometric multiplicities of the eigenvalues of $A$ and $B$ coincide.%
\item[(c)] The numbers of Jordan blocks within the generalized eigenspaces of $A$ and $B$ coincide.%
\end{enumerate}
\end{corollary}

\begin{proof}
We may assume without loss of generality that $A$ and $B$ are given in Jordan normal form with respect to an orthonormal basis of $V$. Lemma \ref{lem_conjugacylifting} yields a topological conjugacy $H:\rmS(V)\rightarrow\rmS(V)$ from $\chi_A$ to $\chi_B$. Let $E_A$ and $E_B$ be the sums of real eigenspaces of $A$ and $B$, respectively. By Proposition \ref{prop_recurrentsetsphere}, the sets $E_A\cap\rmS(V)$ and $E_B\cap\rmS(V)$ are the recurrent sets of $\chi_A$ and $\chi_B$. Hence, $H(E_A\cap\rmS(V))=E_B\cap\rmS(V)$. Define $G:E_A\rightarrow E_B$ by%
\begin{equation*}
  G(x) := \left\{\begin{array}{cc}
                       \|x\| H(\frac{x}{\|x\|}) & \mbox{for } x \in E_A\backslash\{0\},\\
                         0 & \mbox{for } x = 0.
                 \end{array}\right.%
\end{equation*}
The map $G$ is continuous at the origin, since it preserves the norm on $V$. Actually, it is a homeomorphism with inverse%
\begin{equation*}
  G^{-1}(x) := \left\{\begin{array}{cc}
                       \|x\| H^{-1}(\frac{x}{\|x\|}) & \mbox{for } x \in E_B\backslash\{0\},\\
                         0 & \mbox{for } x = 0.
                 \end{array}\right.%
\end{equation*}
The following calculation shows that $G$ is a topological conjugacy from $\varphi_A|_{E_A}$ to $\varphi_B|_{E_B}$. Let $x\in E_A^*$. Then $\|\rme^{At}x\|=\|x\|$, $H(x/\|x\|)\in E_B \cap \rmS(V)$, and%
\begin{eqnarray*}
  G(\rme^{At}x) &=& \|\rme^{At}x\|H\left(\frac{\rme^{At}x}{\|\rme^{At}x\|}\right) = \|x\|\frac{\rme^{Bt}H(\frac{x}{\|x\|})}{\|\rme^{Bt}H(\frac{x}{\|x\|})\|}\\
                &=& \|x\|\rme^{Bt}H\left(\frac{x}{\|x\|}\right) = \rme^{Bt}\left(\|x\| H\left(\frac{x}{\|x\|}\right)\right) = \rme^{Bt}G(x).%
\end{eqnarray*}
By the result of Kuiper \cite{Ku1} and Ladis \cite{Lad}, topological conjugacy of linear flows induced by endomorphisms with purely imaginary eigenvalues is equivalent to linear conjugacy. This implies the assertions (a)--(c).%
\end{proof}

\subsection{Stable Manifolds}%

In this subsection, we consider again a finite-dimensional real vector space $V$ and an endomorphism $A\in\End(V)$ all of whose eigenvalues lie on the imaginary axis. We investigate the stable manifolds of the projective flow $\psi_A$.%

\begin{definition}
Let $\varphi$ be a continuous flow on a compact metric space $(X,d)$. For a point $x\in X$, the set%
\begin{equation*}
  \st(x,\varphi) = \left\{y\in X\ :\ d(\varphi(t,x),\varphi(t,y)) \rightarrow 0 \mbox{ for } t\rightarrow\infty\right\}%
\end{equation*}
is called the \emph{stable manifold} of $x$. We define an equivalence relation on $X$ by%
\begin{equation*}
  x \sim y \quad:\Leftrightarrow\quad y \in \st(x,\varphi)%
\end{equation*}
\end{definition}

It is clear that the stable manifolds and their topological dimensions are invariants of topological conjugacy. In particular, they do not depend on the metric but only on the topology of $X$.%

\begin{proposition}\label{prop_eqclass}
Consider the projective flow $\psi_A$. Every equivalence class $[\P x]_{\sim}$ contains a unique element of $\RC(\psi_A)=\P E_A$.%
\end{proposition}

\begin{proof}
First we prove that each equivalence class contains at most one element of $\RC(\psi_A)$. Let $p,q \in \P E_A$ be two distinct recurrent points of $\psi_A$. Then we find $x,y\in V$ with $p = \P x$, $q = \P y$, and a norm $\|\cdot\|$ such that $\rme^{At}$ acts as an isometry on $E_A$ for every $t\in\R$. Let us further assume that $\|x\|=\|y\|=1$. It is well-known that a metric on $\P(V)$, compatible with the quotient topology, is given by%
\begin{equation*}
  d(\P z, \P w) = \min\left\{\frac{z}{\|z\|} - \frac{w}{\|w\|},\frac{z}{\|z\|} + \frac{w}{\|w\|}\right\}.%
\end{equation*}
Using this metric, we find that%
\begin{equation*}
  d(\psi_A(t,p),\psi_A(t,q)) = \min\left\{x+y,x-y\right\} > 0%
\end{equation*}
for all $t\in\R$, which implies that $p$ and $q$ are not in the same equivalence class.%

To complete the proof, we have to analyze the Jordan structure of $A$. Let $s_{\max}$ be the largest dimension of a complex Jordan subspace of $A$. For each $s\in\{1,\ldots,s_{\max}\}$ let $J_s \subset V$ be the linear subspace defined as the sum of all real Jordan subspaces of dimension $s$ for the real eigenvalue $0$, or dimension $2s$ for a complex conjugate pair of imaginary nonzero eigenvalues. We write%
\begin{equation*}
  x = x_1 \oplus \cdots \oplus x_{s_{\max}}%
\end{equation*}
for the unique decomposition of $x\in V$ with $x_s \in J_s$. With respect to an appropriate basis of $J_s$, the component $x_s$ has coordinates%
\begin{equation*}
  x_s = (x_{s1},\ldots,x_{ss}),\qquad x_{si} \in \R^{d_s},%
\end{equation*}
where $d_s$ denotes the dimension of the sum of all real eigenspaces contained in $J_s$, such that the coordinates of $\rme^{At}x_s$ are given by%
\begin{eqnarray*}
  \left[\rme^{At}x_s\right]_1 &=& \rme^{\widehat{A}t}\left[ x_{s1} + tx_{s2} + \frac{t^2}{2!}x_{s3} + \cdots + \frac{t^{s-1}}{(s-1)!}x_{ss} \right],\\
  \left[\rme^{At}x_s\right]_2 &=& \rme^{\widehat{A}t}\left[ x_{s2} + tx_{s3} + \frac{t^2}{2!}x_{s4} + \cdots + \frac{t^{s-2}}{(s-2)!}x_{s(s-1)} \right],\\
  &\vdots&\\
  \left[\rme^{At}x_s\right]_s &=& \rme^{\widehat{A}t} x_{ss},%
\end{eqnarray*}
for an appropriately defined skew-symmetric matrix $\widehat{A}$. For fixed $x$ and $s$ let $j_s$ be the largest integer with $x_{sj_s}\neq 0$. Then, dividing by $t^{j_s-1}$, one sees that the point $\P x_s \in \P(J_s)$ is equivalent to the recurrent point $\P y_s$, where $y_s$ has coordinates $(x_{sj_s},0,\ldots,0) \in \R^{sd_s}$. For the point $\P x$ one takes $j := \max_{1\leq s \leq s_{\max}}j_s$ and divides by $t^{j-1}$ to see that also $\P x$ is equivalent to a recurrent point $\P y$, where $y$ is the sum of the $y_s$ with $j_s = j$. This concludes the proof.%
\end{proof}

Using the notation of the above proof, we define the projective subspaces%
\begin{equation*}
  \RC_s := \left\{\P x\in \P E_A\ :\ x_{s-1} = 0,\ x_{s-2} = 0,\ldots,x_1=0\right\}%
\end{equation*}
for $s = 1,\ldots,s_{\max}$. Since $\RC_{s+1}\subset\RC_s$, we have a disjoint union%
\begin{equation*}
  \RC(\psi_A) = \bigcup_{s=1}^{s_{\max}}\RC_s\backslash \RC_{s+1}\quad (\RC_{s_{\max}+1}:=\emptyset).%
\end{equation*}

The proof of the following lemma is immediate and will be omitted.%

\begin{lemma}\label{lem_topdim1}
The topological dimension of $\RC_s\backslash\RC_{s+1}$ is $\sum_{i=s}^{s_{\max}}d_i - 1$.%
\end{lemma}

The next lemma gives the correct formulas for the dimensions of the stable manifolds of the recurrent points. The formulas in \cite[Lemma 5]{Ku2} are not quite correct.%

\begin{lemma}\label{lem_topdim2}
The dimension of the stable manifold of some point $p\in\RC(\psi_A)$ is constant in $\RC_s\backslash\RC_{s+1}$ and its value is%
\begin{equation}\label{eq_dimform}
  D_s := n+1 - \sum_{i=s}^{s_{\max}}(i+1-s)d_i,%
\end{equation}
where $n+1$ is the dimension of $V$.%
\end{lemma}

\begin{proof}
First consider the case $s = s_{\max}$. We have%
\begin{equation*}
  \RC_{s_{\max}} = \left\{\P x \in \P E_A\ :\ x_{s_{\max}}\neq 0,\ x_{s_{\max}-1}=\cdots=x_1=0\right\}.%
\end{equation*}
Hence, any point $\P y \in \RC_{s_{\max}}$ satisfies%
\begin{equation*}
  y_{s_{\max}} = (z,0,\ldots,0),\quad y_s = 0 \mbox{ for } s \leq s_{\max}-1%
\end{equation*}
for some $z \in \R^{d_{s_{\max}}}$. From the proof of Proposition \ref{prop_eqclass} it follows that the points $\P x$ which are equivalent to $\P y$ satisfy%
\begin{equation*}
  x_{s_{\max}} = (\underbrace{\star,\ldots,\star}_{j-1 \mbox{ entries}},z,\underbrace{0,\ldots,0}_{s_{\max}-j \mbox{ entries}}).%
\end{equation*}
For the components $x_s$ with $s<s_{\max}$, at most the first $j-1$ components can be different from zero. The dimension of $[\P y]_{\sim}$ is determined by the set of those $\P x$ which satisfy%
\begin{equation*}
  x_{s_{\max}} = (\star,\ldots,\star,z).%
\end{equation*}
Here, the number of real coordinates that can be chosen freely is%
\begin{equation*}
  (s_{\max}-1)d_{s_{\max}} + \sum_{s=1}^{s_{\max}-1} sd_s = \sum_{s=1}^{s_{\max}}sd_s - d_{s_{\max}} = n+1 - d_{s_{\max}}.%
\end{equation*}
This proves the dimension formula for points in $\RC_{s_{\max}}$.%

Now consider a point $\P y \in \RC_{s_{\max}-1}\backslash\RC_{s_{\max}}$. This point satisfies%
\begin{equation*}
  y_{s_{\max}} = (z_1,0,\ldots,0),\quad y_{s_{\max}-1} = (z_2,0,\ldots,0)\neq 0,\quad y_s = 0 \mbox{ for } s \leq s_{\max}-2.%
\end{equation*}
With the same arguments as above, one sees that the maximal number of real coordinates in $J_{s_{\max}-1}\oplus \cdots \oplus J_1$ that can be chosen freely is given by%
\begin{equation*}
  (s_{\max}-2)d_{s_{\max}-1} + \sum_{s=1}^{s_{\max}-2}sd_s = n+1 - s_{\max}d_{s_{\max}} - d_{s_{\max}-1}.%
\end{equation*}
But there are also coordinates in $J_{s_{\max}}$ that can be chosen freely. The number of these coordinates is%
\begin{equation*}
  d_{s_{\max}}(s_{\max}-1) - d_{s_{\max}} = d_{s_{\max}}(s_{\max}-2).% 
\end{equation*}
Altogether we have $n+1 - d_{s_{\max}-1} - 2d_{s_{\max}}$ free coordinates. For arbitrary $s\in\{1,\ldots,s_{\max}\}$, similarly one sees that the number of free coordinates is given by%
\begin{equation*}
  (s-1)d_s + \sum_{i=1}^{s-1}id_i + (s-1)\sum_{i=s+1}^{s_{\max}}d_i = n+1 - \sum_{i=s}^{s_{\max}}(i+1-s)d_i.%
\end{equation*}
This concludes the proof.%
\end{proof}

We illustrate the preceding lemma by a concrete example.%

\begin{example}
Consider the matrix%
\begin{equation*}
  A = \left(\begin{array}{ccccccccc}
               0 & 1 & & & & & & & \\
               & 0 & 1 & & & & & & \\
               & & 0 & & & & & & \\
               & & & 0 & 1 & & & & \\
               & & & & 0 & 1 & & & \\
               & & & & & 0 & & & \\
               & & & & & & 0 & 1 & \\
               & & & & & & &  0 & \\
               & & & & & & &  & 0
            \end{array}\right).%
\end{equation*}
We have $s_{\max} = 3$, $d_3 = 2$, $d_2 = 1$ and $d_1 = 1$. Applying the flow $\rme^{At}$ to some vector $y = \sum_{i=1}^9 y_i e_i \in\R^9$, we obtain%
\begin{equation*}
  \rme^{At}y = \left(\begin{array}{c}
                  y_1 + ty_2 + \frac{t^2}{2}y_3\\
                  y_2 + ty_3\\
                  y_3\\
                  y_4+ty_5+\frac{t^2}{2}y_6\\
                  y_5+ty_6\\
                  y_6\\
                  y_7+ty_8\\
                  y_8\\
                  y_9 \end{array}\right).%
\end{equation*}
The recurrent set of the projective flow $\psi_A$ is given by $\RC(\psi_A) = \{\P y : y_2 = y_3 = y_5 = y_6 = y_8 = 0\}$ and consists of equilibria. Furthermore,%
\begin{equation*}
  \RC_{s_{\max}} = \RC_3 = \left\{\P y : y_i = 0 \mbox{ for all } i \notin \{1,4\} \right\}.%
\end{equation*}
If $y_3\neq0$ and $y_6\neq0$, one sees (dividing by $t^2$) that $\P\rme^{At}y$ converges to the point $p=\P(y_3,0,0,y_6,0,0,0,0,0) \in \RC_{s_{\max}}$. Hence, one can choose the
coordinates $y_i$ with $i\neq\{3,6\}$ freely, which shows that the dimension of the stable manifold of $p$ is $7 = 9 - d_3$, which is consistent with formula \eqref{eq_dimform}. Now, consider the set%
\begin{equation*}
  \RC_2 \backslash \RC_3 = \left\{\P y : y_9 = 0,\ y_7 \neq 0,\ y_2 = y_3 = y_5 = y_6 = y_8 = 0 \right\}.%
\end{equation*}
If $y_3 = y_6 = 0$, one sees (dividing by $t$) that $\P\rme^{At}y$ converges to $p=\P(y_2,0,0,y_5,0,0,y_8,0,0)$. We have $p\in\RC_2\backslash\RC_3$ if and only if $y_2=y_5=0$. Hence, we see that the dimension of the stable manifold of $p$ is $4 = 9 - d_2 - 2d_3$, which is consistent with formula \eqref{eq_dimform}. Finally, consider%
\begin{equation*}
  \RC_1 \backslash \RC_2 = \left\{\P y : y_9 \neq 0,\ y_2 = y_3 = y_5 = y_6 = y_8 = 0\right\}.%
\end{equation*}
If $y_9\neq0$ and $y_2=y_3=y_5=y_6=y_8=0$, then $\P\rme^{At}y$ converges to $p=\P(y_1,0,0,y_4,0,0,y_7,0,y_9)\in\RC_1\backslash\RC_2$. Hence, one cannot choose any coordinates freely, and thus the dimension of the stable manifold is $0 = 9 - d_1 - (2d_2 + 3d_3)$, again consistent with formula \eqref{eq_dimform}.%
\end{example}

The following corollary is the last ingredient for the proof of our classification result.%

\begin{corollary}\label{cor_jordanblocks}
Let $A$ and $B$ be endomorphisms of $V$ all of whose eigenvalues lie on the imaginary axis, such that $\psi_A$ and $\psi_B$ are topologically conjugate. Then the numbers $s_{\max}$ and $d_s$, $s=1,\ldots,s_{\max}$, coincide for $A$ and $B$, that is, the Jordan structures of $A$ and $B$ coincide.%
\end{corollary}

\begin{proof}
Let $h$ denote the homeomorphism which conjugates $\psi_A$ and $\psi_B$. For different values of $s$, the dimension $D_s$ (cf.~Formula \eqref{eq_dimform}) has a different value, in fact we have $D_s > D_{s-1}$ for all $s$, which follows from%
\begin{eqnarray*}
  D_s - D_{s-1} &=& - \sum_{i=s}^{s_{\max}}(i+1-s)d_i + \sum_{i=s-1}^{s_{\max}}(i+1-s+1)d_i\\
                &=& \sum_{i=s-1}^{s_{\max}}d_i \geq d_{s_{\max}} > 0.%
\end{eqnarray*}
Since $h$ maps stable manifolds of $\psi_A$ onto corresponding stable manifolds of $\psi_B$, preserving the dimension, we can thus conclude that the number $\#\{s : d_s \neq 0\}$ is the same for both flows. If $s_{\max}(A) = s_1(A) > \cdots > s_k(A)$ are the corresponding numbers for $A$ with $d_{s_i(A)}(A)\neq0$ and $s_{\max}(B) = s_1(B) > \cdots > s_k(B)$ the ones for $B$, then 
\begin{equation*}
  h\left(\RC_{s_i(A)}\backslash \RC_{s_i(A)+1}\right) = \RC_{s_i(B)}\backslash\RC_{s_i(B)+1} \mbox{\quad for\ } i=1,\ldots,k.%
\end{equation*}
Since $h$ preserves the dimension, Lemma \ref{lem_topdim1} yields%
\begin{equation*}
  \delta_i := d_{s_i(A)}(A) = d_{s_i(B)}(B) \mbox{\quad for\ }  i=1,\ldots,k.%
\end{equation*}
Putting $s = s_2(A)$ and $s = s_2(B)$ in Formula \eqref{eq_dimform} gives%
\begin{equation}\label{eq_sdiff1}
  D_{s_2(A)}(A) = D_{s_2(B)}(B) \quad \Rightarrow \quad s_1(A) - s_2(A) = s_1(B) - s_2(B).%
\end{equation}
Now let us assume that $s_j(A) - s_{j+1}(A) = s_j(B) - s_{j+1}(B)$ holds for $j=1,2,\ldots,l-2$ and proceed by induction on $l$. We have%
\begin{equation*}
  \sum_{i=s_l(A)+1}^{s_1(A)}(i+1-s_l(A))d_i(A) = \sum_{i=s_l(B)+1}^{s_1(B)}(i+1-s_l(B))d_i(B).%
\end{equation*}
This is equivalent to%
\begin{equation*}
  \sum_{j=1}^{l-1}(s_j(A)-s_l(A))\delta_j = \sum_{j=1}^{l-1}(s_j(B)-s_l(B))\delta_j.%
\end{equation*}
For each $j\in\{1,\ldots,l-1\}$ we can write%
\begin{eqnarray*}
  s_j(A) - s_l(A) &=& \sum_{i=j}^{l-1}s_i(A) - \sum_{i=j+1}^l s_i(A)\\
                  &=& \sum_{i=j}^{l-1}(s_i(A) - s_{i+1}(A))\\
                  &=& \sum_{i=j}^{l-2}(s_i(B) - s_{i+1}(B)) + (s_{l-1}(A) - s_l(A))\\
                  &=& (s_j(B) - s_{l-1}(B)) - (s_{l-1}(A) - s_l(A)).%
\end{eqnarray*}
This yields%
\begin{eqnarray*}
  && \sum_{j=1}^{l-1}\left[(s_j(B) - s_{l-1}(B)) - (s_{l-1}(A) - s_l(A))\right]\delta_j \\
  && = \sum_{j=1}^{l-1}\left[(s_j(B) - s_{l-1}(B)) + (s_{l-1}(B) - s_l(B))\right]\delta_j,%
\end{eqnarray*}
which immediately gives $s_{l-1}(A)-s_l(A) = s_{l-1}(B) - s_l(B)$ and hence concludes the induction step. It easily follows that the differences $s_i(A) - s_i(B)$ are constant. Since $\sum_{i=1}^k s_i(\cdot)\delta_i = n+1$ for $(\cdot)\in\{A,B\}$, we obtain%
\begin{equation*}
  0 = \sum_{i=1}^k \left(s_i(A) - s_i(B)\right)\delta_i \quad \Rightarrow \quad s_i(A) - s_i(B) \equiv 0.%
\end{equation*}
The proof is finished.%
\end{proof}

\section{The Main Theorem}\label{sec_mt}

Finally, we can give a full proof of the announced classification result:%

\begin{theorem}\label{thm_hauptsatz}
Let $V$ be a finite-dimensional real vector space and $A,B\in\End(V)$. Then the induced projective flows $\psi_A$ and $\psi_B$ are topologically conjugate on $\P(V)$ if and only if, with respect to individual linear coordinates, $A$ and $B$ can be written in the form%
\begin{eqnarray*}
  A &=& (\lambda_1 \id + \sigma_1) \oplus (\lambda_2\id + \sigma_2) \oplus \cdots \oplus (\lambda_k\id + \sigma_k),\\
  B &=& (\mu_1 \id + \sigma_1) \oplus (\mu_2\id + \sigma_2) \oplus \cdots \oplus (\mu_k\id + \sigma_k),%
\end{eqnarray*}
with real numbers $\lambda_1 > \lambda_2 > \cdots > \lambda_k$, $\mu_1 > \mu_2 > \cdots > \mu_k$, and endomorphisms $\sigma_1,\ldots,\sigma_k$ with eigenvalues lying on the imaginary axis.%
\end{theorem}

\begin{proof}
The proof follows by combining Theorem \ref{thm_projconj} with Proposition \ref{prop_morse}, Corollary \ref{cor_eigenvalues} and Corollary \ref{cor_jordanblocks}. Indeed, Theorem \ref{thm_projconj} settles the direction of the proof which involves the construction of the conjugacy. In the other direction of the proof, Proposition \ref{prop_morse} reduces everything to endomorphisms with eigenvalues whose real parts vanish. Then Corollary \ref{cor_jordanblocks} shows that, up to the eigenvalues, the Jordan structures of both endomorphisms are the same. Finally, a careful application of Corollary \ref{cor_eigenvalues} (using the fact that the sets
$\RC_s\backslash \RC_{s+1}$ and hence the projective subspaces $\RC_s$ are respected by the topological conjugacy, as a consequence of Lemma \ref{lem_topdim2}) shows that the eigenvalues of both endomorphisms are the same and that they are distributed in the right way over the Jordan blocks of different sizes.% 
\end{proof}

\section{Acknowledgements}%

We are indebted to Fritz Colonius for numerous discussions about the mathematical details of this paper. Furthermore, we thank Mauro Patr\~ao for critically reading an earlier version of this paper and suggesting some simplifications.%

\end{document}